\documentclass[11pt, reqno]{amsart}
\usepackage[margin=3.6cm]{geometry}

\usepackage{amsmath, amsfonts, amsthm, amssymb}
\usepackage{exscale}
\usepackage{cite}
\usepackage{epsfig}
\usepackage{amscd}
\usepackage{graphicx}

\usepackage{color}

\renewcommand{\geq}{\geqslant}
\renewcommand{\leq}{\leqslant}

\newtheorem{theorem}{Theorem}

\newtheorem{proposition}{Proposition}[section]

\newtheorem{lemma}[proposition]{Lemma}
\newtheorem*{main-theorem}{Main Theorem}
\newtheorem*{theorem*}{Theorem}

\theoremstyle{definition}

\newtheorem{remark}[proposition]{Remark}

\newtheorem*{remark*}{Remark}

\numberwithin{equation}{section}

\def\phi{\varphi}

\def\reals{{\mathbb R}}

\def\Ci{{\mathcal C}^\infty}

\def\supp{\mathrm{supp}\,}

\def\O{{\mathcal O}}
\def\SS{{\mathbb S}}

\def\phi{\varphi}

\def\dist{\text{dist}\,}

\def\be{\begin{eqnarray*}}
\def\ee{\end{eqnarray*}}
\def\ben{\begin{eqnarray}}
\def\een{\end{eqnarray}}

\def\L2R{L_{\text{Rest}}^2}

\def\11{\mathds{1}}

\def\L2c{L^2_{\text{comp}}}

\def\p{\partial}

\def\bu{\bar{u}}

\begin{document}

\title[One point non-concentration]{A one point non-concentration
  estimate for Laplace eigenfunctions on polygons}

\author[H. Christianson]{Hans Christianson}
\address[H. Christianson]{ Department of Mathematics, University of North Carolina.\medskip}
 \email{hans@math.unc.edu}

%
%
%
\begin{abstract}

  In this paper we consider  eigenfunctions of the Laplacian
  on a planar domain with polygonal boundary with Dirichlet, Neumann,
  or mixed boundary conditions.  The main result is a quantitative
  estimate on the $L^2$ mass of eigenfunctions near a point in terms
  of the distance to the nearest non-adjacent boundary face.  In
  particular, eigenfunctions cannot concentrate completely at any one
  single point.  
  The technique of proof is to
  use the commutator ideas from the recent work of the author \cite{Chr-tri,Chr-simp}
 on
 triangles and simplices.


\end{abstract}

\maketitle

\section{Introduction}
In this paper, we study the distribution of interior $L^2$ mass of
Laplace eigenfunctions on polygonal domains with Dirichlet, Neumann,
or mixed boundary conditions.  Let $\Omega \subset
\reals^2$ be a bounded open set.  Eigenfunctions on $\Omega$ are used
to model, for example, the fundamental modes of vibration for a drum
with shape $\Omega$, as well as other physical phenomena.  Eigenfunctions are highly sensitive to the
geometry of the boundary of $\Omega$ and the boundary conditions
imposed on $\p \Omega$.  This is part of the
``classical-quantum correspondence''.  The classical problem in a
planar domain is to consider the trajectories of a billiard ball on a
table shaped like $\Omega$.  An ideal billiard ball will follow a straight
line until it meets the boundary, at which point it will reflect
according to Snell's law (angle in equals angle out).   If the boundary of $\Omega$ is
sufficiently smooth, then one can describe the billiard trajectories
as a map on the closed co-ball bundle $\overline{B^* \p \Omega}$, by specifying the point of impact on the
boundary and the incoming direction.  
Consider now a wave on $\Omega$.
Waves tend to travel in packets along straight lines in planar domains
as well, and
reflect off boundary walls according to Snell's law.  But wave packets cannot be completely localized to a
single billiard ball trajectory, so they can do crazy things when they
reflect off a wall, and the curvature and regularity of the boundary
of $\Omega$ at the reflection  point can cause wave packets to focus,
de-focus, disperse, diffract, glance, and many other possibilities.

If $\p \Omega$ is a closed polygonal path, then at each corner the
boundary has only Lipschitz regularity, while away from the corners,
the boundary is affine, so $\Ci$.  Imagining a billiard ball on a
polygonal domain $\Omega$, one begins to see subtleties even in the
classical problem.  How does one specify how a billiard ball reflects
when it heads into a corner?

By separation of variables, the study of solutions to the wave
equation on $\Omega$ can be reduced to studying eigenfunctions.  In
this paper, we consider the following eigenfunction problem.  Let
$\{u_j \}$ be a sequence of functions satisfying
\begin{equation}
  \label{E:ef-1}
\begin{cases}
  - \Delta u_j = \lambda_j^2 u_j, \,\, \text{in } \Omega, \\
  B u_j = 0 \text{ on } \p \Omega, \\
  \| u_j \|_{L^2 ( \Omega )} = 1.
\end{cases}
\end{equation}
Here $B$ is a boundary operator, $Bu_j = u_j$ or $B u_j = \p_\nu u_j$
on each affine segment of $\p \Omega$ (see Subsection \ref{SS:notation} for a rigorous
definition).  
It is classical that this problem has a countably infinite number of
solutions, from which one obtains an orthonormal basis for $L^2$.
With this sequence of eigenfunctions comes a sequence of eigenvalues
$\lambda^{2}_j$, with $\lambda_j \to \infty$.  The size of the $\lambda^{-1}_j$s corresponds to
the wavelength, so $\lambda_j$ is proportional to the frequency of
oscillation of associated waves.  Hence the limit as $\lambda_j \to \infty$ describes ``high-frequency''
behaviour of eigenfunctions.

The goal of this paper is to study how eigenfunctions can concentrate
(or more precisely, {\it not} concentrate)
in the limit $\lambda_j \to \infty$.  Non-concentration results have a
huge history, which we briefly discuss below in Subsection \ref{SS:History}.    Many of the techniques used in these previous works
make use of {\it microlocal analysis}, which is the study of waves in
both space and frequency, using knowledge about the associated
classical problem.  This means that to study eigenfunctions using these
tools, a certain amount of regularity of the boundary is necessary.
This is so that the classical billiard map is well-defined, and the
billiard flow is sufficiently smooth that quantum observables can be
constructed roughly as functions of the billiard flow.

Of course in the case of polygonal domains, the billiard flow is
generally not
smooth, and many techniques from microlocal analysis break down.
In the present
work, we bypass the use of microlocal analysis in favor of commutator
methods.  This  has the virtue of being very robust and requiring very
few assumptions on the regularity of the boundary.  However, it has
the defect of losing any information about the {\it frequency}
localization of eigenfunctions.

The main result in this paper 
is a quantitative estimate on the $L^2$ mass of a sequence of
eigenfunctions in a neighbourhood of a single point.
A novelty of this work is that we also obtain a quantitative estimate
in a neighbourhood of a boundary point as well.


\begin{theorem}
  \label{T:poly-nc}
   Let $\Omega \subset \reals^2$ be a connected, bounded, and open set with polygonal boundary, and let $\{ u_j \}$ be a sequence of  eigenfunctions
on $\Omega$ satisfying \eqref{E:ef-1}.
Let $p_0 \in  \overline{\Omega}$ be any point (including on the
boundary), and
let
$d = \dist (p_0, F')>0$, where $F' \subset \p \Omega$ is the nearest
boundary face not adjacent to $p_0$.  Then for each $0 < \alpha <1$, 
\begin{equation}
  \label{E:interior-est-0}
  \limsup_{\lambda_j \to \infty} \| u \|^2_{L^2 ( D(p_0, \alpha d  ) ) } \leq
  \frac{1}{2-\alpha}.
\end{equation}

\end{theorem}

\begin{remark}
Put another way, Theorem \ref{T:poly-nc} says that any semiclassical
defect measure associated to the sequence $\{ u_j \}$ cannot be
supported at only a single point, including boundary points.

  \end{remark}

\begin{remark}
The lower bound on the mass outside of a neighbourhood of $p_0$ independent of $\lambda_j$   is very strong.  In general, given any
open, bounded subset  $U \subset \Omega$, Carlemann estimates give an
exponential lower bound
\[
\| u_j \|_{L^2(U)} \geq c e^{-c \lambda_j}
\]
for some $c>0$.  The tradeoff is that for our result $U$ is a large
subset; the complement of a neighbourhood of a single point.  In
particular, our theorem does not rule out the possibility of glancing
modes, which can concentrate very strongly in a neighbourhood of size
$\lambda_j^{-2/3}$ of the boundary due to Airy asymptotics of glancing modes
(see \cite{AS-book,Mel-Gl-I} for a discussion of Airy asymptotics and
glancing, and \cite{CHT-ND} for a discussion of optimality).


  \end{remark}

\begin{remark}
In this paper, we consider only {\it classical} polygons, in the sense
that the boundary is a closed, piecewise affine curve with finitely
many corners.  Each corner makes an interior angle $\theta$ satisfying $0 <
\theta < 2\pi$.

If $p_0 \in \p \Omega$, then $p_0$ is either on the interior of an
open face or it is a corner.  If $p_0$ is a corner making an angle $0
< \theta < \pi$, we call it a {\it convex} corner, and if the angle is
$\pi < \theta < 2 \pi$, we call it a {\it concave} corner.  In all cases,
we prove a {\it quantitative} estimate in terms of distance to the
nearest non-adjacent side.  
  This is stated concretely in
Propositions \ref{P:interior} and \ref{P:flat-side}.

  \end{remark}

\subsection{History}
\label{SS:History}
Non-concentration type estimates have an enormous history for several
reasons.  Research into properties of eigenfunctions is an old subject
- for example Fourier series are eigenfunction expansions.  The
connections to physical phenomena and other areas of mathematics
(acoustics, scattering, tunneling, elliptic equations, 
quantum chaos, number theory, etc.) make qualitative properties of
eigenfunctions an important and continuing area of research.  Many
questions about eigenfunctions are relatively concrete to state,
allowing even undergraduate students to at least understand the
questions.  (And prove new interesting results!  The author has
several students working on related problems.) 

The strongest non-concentration is
equidistribution in phase space.  If the classical billiard flow is
{\it ergodic} (roughly ``chaotic''), then the eigenfunctions are known
to equidistribute in phase space (at least along a density one
subsequence) \cite{Sni,Zel1,CdV}.  That means, except possibly for a
density zero subsequence, the eigenfunctions concentrate equally
everywhere.  This is called {\it quantum ergodicity}.

These possible exceptional subsequences of {\it non}-quantum ergodic
eigenfunctions have also been heavily studied.  In the case of joint
Hecke-Laplace eigenfunctions Lindenstrauss proved there are no
exceptional subsequences, a property called quantum unique ergodicity \cite{Lin-que}.  In the case of the Bunimovich stadium
Hassell proved that as one varies the length of the rectangular part
of the stadium, quantum unique ergodicity fails with probability 1
\cite{Has-nonque}.  In the works \cite{CdVP-I,CdVP-II,Chr-NC,Chr-QMNC}, it is
shown that eigenfunctions cannot concentrate too sharply along an
unstable periodic geodesic.  That gives a restriction on what possible
limit measures exist for sequences of eigenfunctions in this case.

On the other hand, as mentioned above, it is possible for
eigenfunctions to concentrate very sharply in a $\lambda$ dependent
neighbourhood of a hypersurface, or even near a single point.
Further, by measuring $L^p$ norms
instead of $L^2$ norms, even more refined
concentration/non-concentration estimates are known \cite{Sog-sharm,Sog-lp}.
It should be noted that these estimates are sharp on the sphere, which
clearly plays no role in the present paper!

Yet another measure of concentration/non-concentration is to consider
restrictions of eigenfunctions to lower dimensional sets.  Although
this is not the topic of this paper, it is worthwhile to mention a few
results.  Burq-G{\'e}rard-Tzvetkov \cite{BGT-erest} give sharp upper bounds on restrictions of
eigenfunctions and the author's work with Hassell-Toth \cite{CHT-ND}
gives sharp upper bounds on the Neumann data on hypersurfaces.  In the
setting of quantum ergodic eigenfunctions, more is known
\cite{GeLe-qe,HaZe,ToZe-1,Toze-2,CTZ-1,DyZw-QER}, however again these
results require smoothness of the domain and a priori knowledge about
ergodicity of the classical flow.  The novel feature of 
the author's work on boundary
values on triangles and simplices \cite{Chr-tri,Chr-simp} is that it
does not require any dynamical systems knowledge, and only relies on
integrations by parts.  This is the starting point for the present paper.

\subsection*{Acknowledgements}
The author would like to thank Luc Hillairet, Jeremy Marzuola, and
John Toth for very helpful and interesting discussions about the
topics in this paper.  This work is supported in part by NSF grant DMS-1500812.

\section{Preliminaries}
\label{S:preliminaries}
 Here we record some notation and basic facts  needed to complete the
 quantitative estimates in Propositions \ref{P:interior} and
 \ref{P:flat-side} below.

\subsection{Notation}
\label{SS:notation}
In this subsection, we define more precisely what is meant by
polygonal boundary and the boundary operator $B$ used in Theorem \ref{T:poly-nc}.

Let $\Gamma = \cup_{j = 1}^N \Gamma_j \subset \reals^2$ be a union of disjoint closed simple
polygonal curves, and let $\Omega \subset \reals^2$ be the bounded
open domain enclosed by $\Gamma$.  In order to have consistent
normalization for the eigenfunctions, let us assume that $\Omega$ is
connected.  For each $1 \leq j \leq N$,
$\Gamma_j$ is a union of a finite number of linear segments $\Gamma_j
= 
\cup_{k_j = 1}^{M_j} \Gamma_{k_j}$.  On each segment, let $\nu_{k_j}$
denote the outward (with respect to $\Omega$) unit normal vector, and
let $B_{k_j}$ be
a homogeneous boundary operator:
\[
B_{k_j} u = \begin{cases}
  u|_{\Gamma_{k_j}} , \text{ for Dirichlet boundary conditions, or}
  \\
  \p_{\nu_{k_j}} u |_{\Gamma_{k_j}}  \text{ for Neumann boundary
    conditions}.
\end{cases}
\]
Let $B = \sum_{j=1}^N \sum_{k_j = 1}^{M_j} B_{k_j}$ denote the total
boundary operator.

Let us drop the
subscript $j$ notation on $\lambda_j$ and rescale
$h = \lambda^{-1}$ so that the solutions of \eqref{E:ef-1} satisfy the
semiclassical eigenfunction problem
\begin{equation}
  \label{E:ef-2}
\begin{cases}  -h^2 \Delta u = u \text{ in } \Omega, \\
  Bu = 0 \text{ on } \p \Omega, \\
  \| u \|_{L^2 ( \Omega) } = 1.
\end{cases}
\end{equation}
We are then interested in asymptotics as $h \to 0$, and when we write
$u$, we implicitly mean $u = u(h)$ is a sequence of eigenfunctions
depending on $h$.

If $p_0 \in \overline{\Omega}$, then $p_0$ is either in the interior
of $\Omega$, a boundary point on the interior of an edge, a boundary
point at a convex corner, or a boundary point at a concave corner.
These cases are discussed in Sections \ref{SS:interior}-\ref{SS:int-side}.

\subsection{Regularity of eigenfunctions on polygons}
Before continuing, we remark briefly about regularity of
eigenfunctions, which will allow us to perform the necessary
integrations by parts.
The book of Grisvard \cite{Grisvard-book} has a very detailed account
of regularity for elliptic equations on non-smooth domains.  In
\cite{Chr-simp}, the author studied eigenfunctions on simplices, which
are convex and \cite{Grisvard-book} contains the necessary results for
that paper.  In the present paper, in general a polygon is not convex
(and indeed we study eigenfunctions near concave corners), however
even more detailed information is available in \cite{Grisvard-book}
for polygons.  In particular, polygons are Lipschitz domains, and
\cite[Theorem 1.4.4.6]{Grisvard-book} gives continuity of the
derivative from $H^m$ to $H^{m-1}$, as long as $m \neq 1/2$.  Theorem
4.3.1.4 in \cite{Grisvard-book} gives elliptic regularity estimates on
polygons, so that we can conclude that eigenfunctions are in $H^m$ for
each $m \geq 0$.  Then any derivatives of eigenfunctions are in $H^1$,
so \cite[Theorem 1.5.3.1]{Grisvard-book} shows we can integrate by
parts using Green's formula.   We finally observe that, by the same
results, multiplying an eigenfunction by a smooth bounded function
does not decrease its regularity.

\subsection{The radial vector field}
Here we recall some identities involving the radial vector
field $r \p_r$ (in polar coordinates).  The reason for this is two-fold.
First, commuting with the Laplacian reproduces the Laplacian.  The
second is that, fixing a point $p_0 \in \p \Omega$ and translating so
that $p_0 = 0$, the vector field $r \p_r$ is tangential to any
segments emanating from $0$.  This is especially convenient when applied to
functions with Dirichlet boundary conditions $u|_{\p \Omega} = 0$, as
then locally $r \p_r u = 0$ on the boundary since this is a tangential derivative of $0$.
What is not so obvious is that the radial vector field is also
well-behaved when applied to Neumann or mixed boundary conditions.
This is discussed in Lemma \ref{L:zero} in Section \ref{SS:int-side} below.

For completeness, 
let us  state the well known commutator and change of
variable formulae used in this paper.
\begin{lemma}
  \label{L:reproduce}
Let $\reals^n = (r, \theta)$ be polar coordinates with $ 0 \leq r <
\infty$ and $\theta \in \SS^{n-1}$.  The polar Laplacian is
\[
-\Delta = -\p_r^2 - \frac{(n-1)}{r} \p_r - \frac{1}{r^2}
\Delta_{\SS^{n-1}},
\]
where $-\Delta_{\SS^{n-1}}$ is the Laplacian on the unit $n-1$
sphere.  Let $X = r \p_r$.  Then
\[
  [ - \Delta , X] = -2 \Delta.
  \]

  Now let $\reals^n = (x_1, \ldots , x_n )$ be rectangular coordinates
  with $x_j \in \reals$ for each $j$.  The Laplacian in rectangular
  coordinates is $-\Delta = - \p_{x_1}^2 - \ldots - \p_{x_n}^2$, and
  the radial vector field is $X = x_1 \p_{x_1} + \ldots + x_n
  \p_{x_n}$, so that $[-\Delta ,X] = -2 \Delta.$

  \end{lemma}
The proof is a standard computation and change of variables.

The reason for recording these two coordinate versions is that
sometimes it is more convenient to use one or the other.  In this
paper, for interior points, we are integrating over a disc, so using a
vector field which is independent of $\theta$ is reasonable.  As
mentioned above, the radial vector field is tangential to any segment
emanating from $0$, which is convenient for Dirichlet boundary
conditions on a polygon.  But we are also interested in mixed boundary conditions,
for which it seems easier to use rectangular coordinates.

\section{Proof of Theorem \ref{T:poly-nc} for interior points}
\label{SS:interior}
In this Section, we study non-concentration near a point on the
interior of $\Omega$.  It has been known for some time that
eigenfunctions cannot concentrate away from the boundary of a polygon
\cite{HHM,MaRu} (at least for a subsequence of density one), but we provide a proof in the case of a single point as the
estimates are quantitative in this case.  In the notation from Section
\ref{S:preliminaries}, the statement of Theorem \ref{T:poly-nc} is given in the following
Proposition.  
\begin{proposition}
  \label{P:interior}
Let $u$ be a solution of \eqref{E:ef-2}, and let $p_0 \in \Omega$.  Let
$d = \dist (p_0, \p \Omega)>0$.  For each $0 < \alpha <1$, 
\begin{equation}
  \label{E:interior-est}
  \limsup_{h \to 0+} \| u \|^2_{L^2 ( D(p_0, \alpha d  ) ) } \leq
  \frac{1}{2-\alpha}.
  \end{equation}

\end{proposition}

    \begin{figure}
\centering
    {
\begingroup%
  \makeatletter%
  \providecommand\color[2][]{%
    \errmessage{(Inkscape) Color is used for the text in Inkscape, but the package 'color.sty' is not loaded}%
    \renewcommand\color[2][]{}%
  }%
  \providecommand\transparent[1]{%
    \errmessage{(Inkscape) Transparency is used (non-zero) for the text in Inkscape, but the package 'transparent.sty' is not loaded}%
    \renewcommand\transparent[1]{}%
  }%
  \providecommand\rotatebox[2]{#2}%
  \newcommand*\fsize{\dimexpr\f@size pt\relax}%
  \newcommand*\lineheight[1]{\fontsize{\fsize}{#1\fsize}\selectfont}%
  \ifx\svgwidth\undefined%
    \setlength{\unitlength}{322.03591919bp}%
    \ifx\svgscale\undefined%
      \relax%
    \else%
      \setlength{\unitlength}{\unitlength * \real{\svgscale}}%
    \fi%
  \else%
    \setlength{\unitlength}{\svgwidth}%
  \fi%
  \global\let\svgwidth\undefined%
  \global\let\svgscale\undefined%
  \makeatother%
  \begin{picture}(1,0.76492275)%
    \lineheight{1}%
    \setlength\tabcolsep{0pt}%
    \put(0,0){\includegraphics[width=\unitlength,page=1]{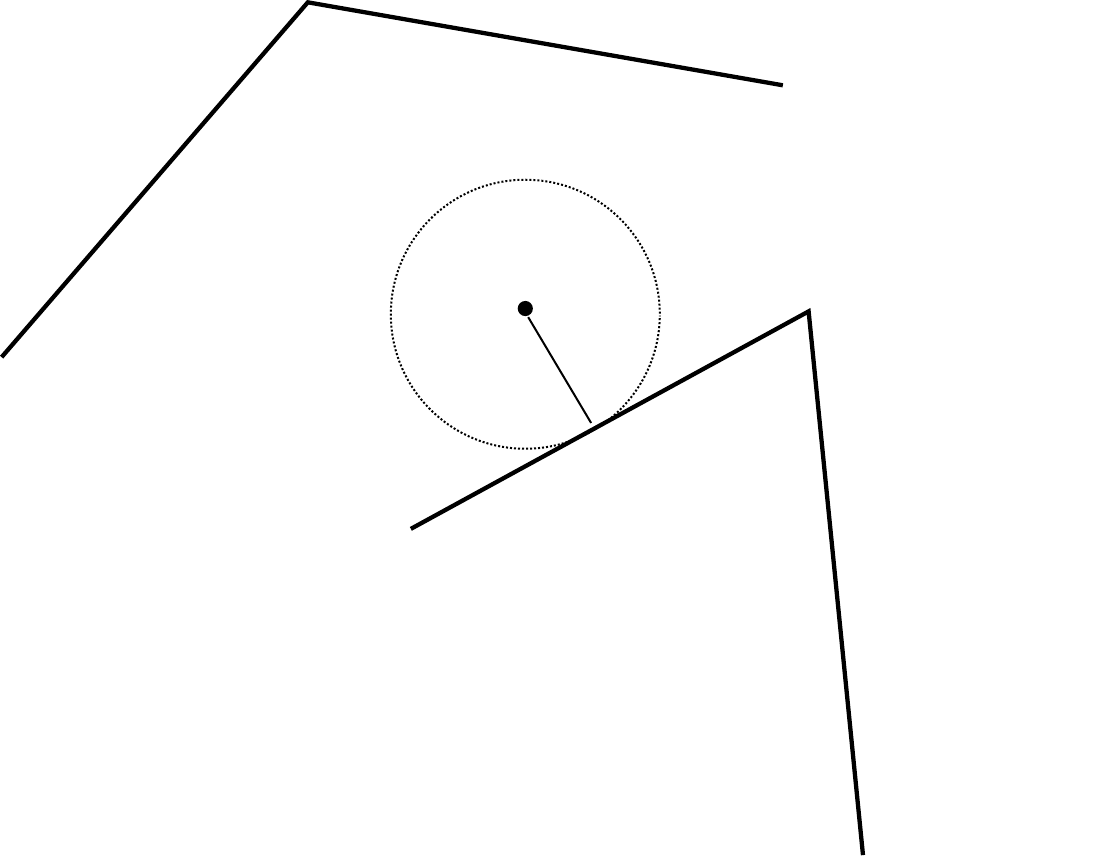}}%
    \put(0.81529498,0.57346606){\color[rgb]{0,0,0}\makebox(0,0)[lt]{\lineheight{1.25}\smash{\begin{tabular}[t]{l}$\p \Omega$\end{tabular}}}}%
    \put(0.50306447,0.49924732){\color[rgb]{0,0,0}\makebox(0,0)[lt]{\lineheight{1.25}\smash{\begin{tabular}[t]{l}$p_0$\end{tabular}}}}%
    \put(0.44676056,0.41223226){\color[rgb]{0,0,0}\makebox(0,0)[lt]{\lineheight{1.25}\smash{\begin{tabular}[t]{l}$d$\end{tabular}}}}%
    \put(0,0){\includegraphics[width=\unitlength,page=2]{poly-int.pdf}}%
  \end{picture}%
\endgroup%
}
\caption{\label{F:int}  The polygon $\Omega$ near $p_0$
 when $p_0$ lies in the interior of $\Omega$.  The distance $d$
  is to the closest point on the boundary.  Here $\p \Omega$
is in bold, and the circle of radius $d $ is in
dashes.  We integrate over a disc of radius $\alpha d$ with $0 <
\alpha < 1$ .}

\end{figure}

\begin{proof}
We argue by contradiction.  Suppose \eqref{E:interior-est} is false.
Then there exists $\alpha>0$ and a subsequence $\{ h_j \}$ as $h \to 0$ such
that
\[
 \| u(h_j) \|^2_{L^2 ( D(p_0, \alpha d  ) ) } \to R >
\frac{1}{2-\alpha}.
\]

Now choose $ \epsilon >0$ satisfying $\epsilon \leq
\min(\frac{d(1 - \alpha)}{4}, \frac{1}{2} )$.  Shrinking $\epsilon>0$ further if
necessary, we may assume 
\[
\lim_{j \to \infty} \| u(h_{j}) \|^2_{L^2 ( D(p_0, \alpha d ) ) } \geq  \frac{1}{2-\alpha} + d \epsilon.
\]
For ease in exposition, let us immediately drop the subscript and
subsequence notation, and just consider a sequence of eigenfunctions
$u$ satisfying
\[
 \| u \|^2_{L^2 ( D(p_0, \alpha d ) ) } \geq  \frac{1}{2-\alpha} + d \epsilon + o(1)
\]
as $h \to 0$.

Translating $\Omega$, we may assume $p_0 = 0$.  We introduce polar
coordinates $(r, \theta)$ near $0$, so that our norm above is over $\{
r \leq \alpha d \}$.  Let $\phi = \phi(r)$ be the function
$\phi_1$ described in Lemma \ref{L:phi1} with $\delta_1 = \alpha d$,
$\delta_2 = d$, and with $\epsilon>0$ specified above, so that
$\phi(r) \equiv 1$ for $r \leq \alpha d + \epsilon^3$,
$\phi(r) \equiv 0$ for $r \geq d - \epsilon^3$.  Lemma
\ref{L:phi1} then guarantees we can choose such a $\phi(r)$ so that
\[
| \phi'| \leq \frac{2}{(1 - \alpha)d} + \epsilon.
\]
Choose also $\psi (r) \in \Ci ( \reals )$ satisfying $\psi \equiv 1$
on $\supp \phi'$, $0 \leq \psi \leq 1$, and $\supp \psi(r) \subset [\alpha d, d ] $.

Now in polar coordinates, $-h^2 \Delta = - h^2 \p_r^2 - h r^{-1} h
\p_r - h^2 r^{-2} \p_\theta^2$, and $[-h^2 \Delta, r \p_r ] = - 2 h^2
\Delta$.  Hence we have
\begin{align}
  \int_\Omega  \phi ([ -h^2 \Delta -1, r \p_r ] u )\bu dV 
  & = -2 \int_\Omega
  \phi (h^2 \Delta u) \bu dV \notag \\
  & = 2 \int_\Omega \phi | u |^2 dV \notag\\
  & \geq 2 \left(  \frac{1}{2-\alpha} + d \epsilon \right) +o(1)\label{E:interior-lower},
\end{align}
since $\phi \equiv 1$ on $D\left(p_0, \alpha d \right)$ and $\phi
\geq 0$ everywhere.  On the
other hand, 
\begin{align*}
  \int_\Omega & \phi ([-h^2 \Delta -1, r \p_r ] u) \bu dV \\
  & = \int_\Omega \phi (( - h^2 \Delta -1) r \p_r u) \bu dV -
  \int_\Omega \phi r (\p_r ( - h^2 \Delta -1)  u )\bu dV \\
& = 
   \int_\Omega \phi ( ( - h^2 \Delta -1) r \p_r u) \bu dV
  \\
  & = \int_\Omega (r \p_r u ) ((-h^2 \Delta -1 ) \phi \bu) dV.
\end{align*}
Here we have used the eigenfunction equation \eqref{E:ef-2} and that
$\phi$ has compact support inside $\Omega$ so there are no boundary terms
when integrating by parts.  The last term also vanishes up to
commuting with $\phi$, so we have
\begin{align*}
  \int_\Omega & \phi ([-h^2 \Delta -1, r \p_r ] u) \bu dV \\
  & = \int_\Omega (r \p_r u ) ([-h^2 \Delta,  \phi] \bu) dV \\
  & = \int_\Omega ( r \p_r u ) ( -2 h \phi' h \p_r \bu ) dV + \O ( h)
  \\
  & \leq 2 \sup | \phi' | \int \psi r | h \p_r u |^2 r dr d \theta + \O(h)
  \\
  & \leq 2 \left( \frac{1}{(1 - \alpha)d } + \epsilon \right) d \int \psi | h \p_r u |^2 r dr d \theta + \O(h).
\end{align*}
To estimate the last term, we use
\begin{align*}
  \int \psi | h \p_r u |^2 r dr d \theta & \leq \int \psi (| h \p_r u
  |^2  + | r^{-1} h \p_\theta u |^2) r dr d \theta \\
  & = \int \psi ( -h^2 \Delta u) \bu dV + \O(h) \\
  & = \int \psi |u|^2 dV + \O(h) \\
  & \leq \left(1 - \frac{1}{2 - \alpha} - d \epsilon \right) + o(1).
\end{align*}
Hence
\begin{align}
    \int_\Omega & \phi ([-h^2 \Delta -1, r \p_r ] u) \bu dV \notag \\
  & \leq 2 \left( \frac{1}{(1 - \alpha) d } + \epsilon \right) d  \int
    \psi | h \p_r u |^2 r dr d \theta + o(1) \notag \\
& \leq  2 d \left( \frac{1}{(1 - \alpha) d} +  \epsilon \right) \left(1 - \frac{1}{2
      - \alpha} - d \epsilon \right) +
    o(1) \notag \\
    & \leq \frac{2}{1-\alpha} + 2d \epsilon - \frac{2}{(1 - \alpha) (
      2 - \alpha)} - \frac{2d \epsilon}{1 - \alpha}  \notag +o(1)\\
    & \leq \frac{2}{1 - \alpha}  - \frac{2}{(1 - \alpha)(2 - \alpha)}
    + o(1) \notag \\
    & \leq \frac{2}{2 - \alpha} + o(1).  \label{E:interior-upper}
\end{align}

Combining \eqref{E:interior-lower} and \eqref{E:interior-upper}, we have
\begin{align*}
  2 \left(  \frac{1}{2 - \alpha} + d \epsilon \right) & \leq \int_\Omega  \phi ([-h^2
    \Delta -1, r \p_r ] u) \bu dV  \\
  & \leq \frac{2}{2 - \alpha} + o(1),
\end{align*}
which is a contradiction for $h >0$ sufficiently small.

  \end{proof}



\section{Proof of Theorem \ref{T:poly-nc} for boundary points}
\label{SS:int-side}

For a boundary point $p_0 \in \p \Omega$, $p_0$ either lies in the
interior of a flat face, or at a corner.  If at a corner, it can
either be a convex or concave corner.  But it turns out that the proof works
more or less the same for all three cases.  In fact, the proof is
nearly identical to the proof of the interior case.  Let $\theta_0$ be  the
angle of $\p \Omega$ at $p_0$, measured from the interior of
$\Omega$.  If $0 < \theta_0 < \pi$, $p_0$ lies at a convex corner, and
if $\pi < \theta_0 < 2 \pi$, $p_0$ lies at a concave corner.  If $\theta_0
= \pi$, then $p_0$ lies on the interior of a face of $\p \Omega$.
The statement of Theorem \ref{T:poly-nc} for boundary points in the
notation of Section \ref{S:preliminaries} is given in the following Proposition.



    \begin{figure}
\centering
    {
\begingroup%
  \makeatletter%
  \providecommand\color[2][]{%
    \errmessage{(Inkscape) Color is used for the text in Inkscape, but the package 'color.sty' is not loaded}%
    \renewcommand\color[2][]{}%
  }%
  \providecommand\transparent[1]{%
    \errmessage{(Inkscape) Transparency is used (non-zero) for the text in Inkscape, but the package 'transparent.sty' is not loaded}%
    \renewcommand\transparent[1]{}%
  }%
  \providecommand\rotatebox[2]{#2}%
  \newcommand*\fsize{\dimexpr\f@size pt\relax}%
  \newcommand*\lineheight[1]{\fontsize{\fsize}{#1\fsize}\selectfont}%
  \ifx\svgwidth\undefined%
    \setlength{\unitlength}{304.63326645bp}%
    \ifx\svgscale\undefined%
      \relax%
    \else%
      \setlength{\unitlength}{\unitlength * \real{\svgscale}}%
    \fi%
  \else%
    \setlength{\unitlength}{\svgwidth}%
  \fi%
  \global\let\svgwidth\undefined%
  \global\let\svgscale\undefined%
  \makeatother%
  \begin{picture}(1,0.75755698)%
    \lineheight{1}%
    \setlength\tabcolsep{0pt}%
    \put(0,0){\includegraphics[width=\unitlength,page=1]{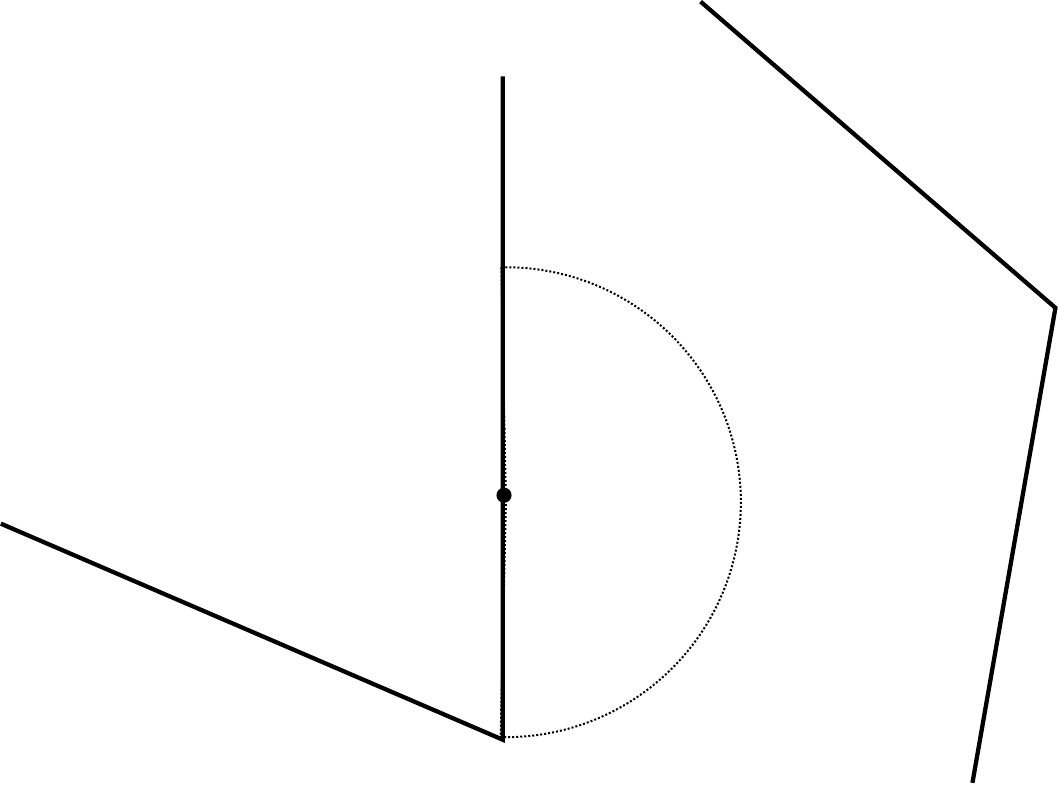}}%
    \put(0.63282266,0.00630294){\color[rgb]{0,0,0}\makebox(0,0)[lt]{\lineheight{1.25}\smash{\begin{tabular}[t]{l}$\p \Omega$\end{tabular}}}}%
    \put(0.35270829,0.28951313){\color[rgb]{0,0,0}\makebox(0,0)[lt]{\lineheight{1.25}\smash{\begin{tabular}[t]{l}$p_0$\end{tabular}}}}%
    \put(0.52095293,0.39416212){\color[rgb]{0,0,0}\makebox(0,0)[lt]{\lineheight{1.25}\smash{\begin{tabular}[t]{l}$d$\end{tabular}}}}%
    \put(0,0){\includegraphics[width=\unitlength,page=2]{poly-flat.pdf}}%
  \end{picture}%
\endgroup%
}
\caption{\label{F:flat}  The polygon $\Omega$ near $p_0 \in \p
  \Omega$ when $p_0$ lies on the interior of a side.  The distance $d$
  is to the closest vertex or closest other side, whichever is
  closer.  Here $\p \Omega$
is in bold, and the circle of radius $d$ is in
dashes. }
    \end{figure}

    \begin{figure}
\centering
    {
\begingroup%
  \makeatletter%
  \providecommand\color[2][]{%
    \errmessage{(Inkscape) Color is used for the text in Inkscape, but the package 'color.sty' is not loaded}%
    \renewcommand\color[2][]{}%
  }%
  \providecommand\transparent[1]{%
    \errmessage{(Inkscape) Transparency is used (non-zero) for the text in Inkscape, but the package 'transparent.sty' is not loaded}%
    \renewcommand\transparent[1]{}%
  }%
  \providecommand\rotatebox[2]{#2}%
  \newcommand*\fsize{\dimexpr\f@size pt\relax}%
  \newcommand*\lineheight[1]{\fontsize{\fsize}{#1\fsize}\selectfont}%
  \ifx\svgwidth\undefined%
    \setlength{\unitlength}{263.77104378bp}%
    \ifx\svgscale\undefined%
      \relax%
    \else%
      \setlength{\unitlength}{\unitlength * \real{\svgscale}}%
    \fi%
  \else%
    \setlength{\unitlength}{\svgwidth}%
  \fi%
  \global\let\svgwidth\undefined%
  \global\let\svgscale\undefined%
  \makeatother%
  \begin{picture}(1,1.12445209)%
    \lineheight{1}%
    \setlength\tabcolsep{0pt}%
    \put(0,0){\includegraphics[width=\unitlength,page=1]{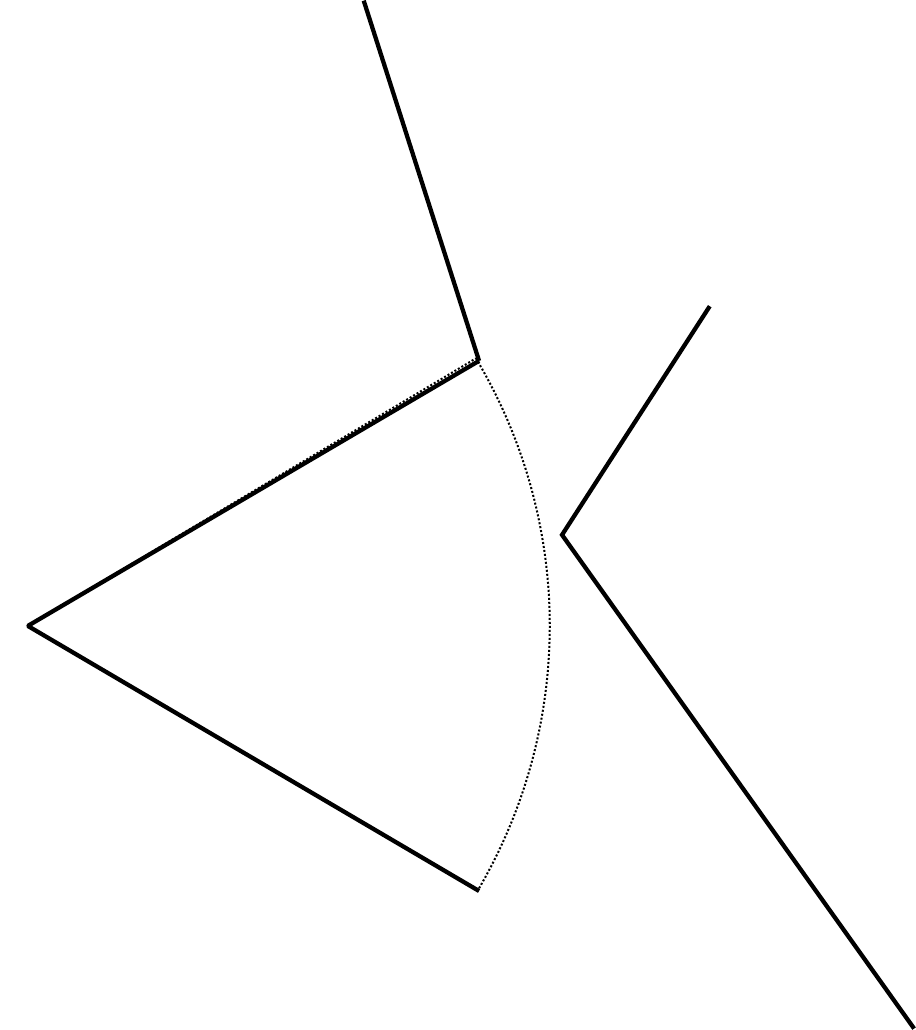}}%
    \put(0.64680121,0.96245241){\color[rgb]{0,0,0}\makebox(0,0)[lt]{\lineheight{1.25}\smash{\begin{tabular}[t]{l}$\p \Omega$\end{tabular}}}}%
    \put(0,0){\includegraphics[width=\unitlength,page=2]{poly-convex.pdf}}%
    \put(0.12758814,0.65313397){\color[rgb]{0,0,0}\makebox(0,0)[lt]{\lineheight{1.25}\smash{\begin{tabular}[t]{l}$d$\end{tabular}}}}%
    \put(0.27120028,0.42556396){\color[rgb]{0,0,0}\makebox(0,0)[lt]{\lineheight{1.25}\smash{\begin{tabular}[t]{l}$S$\end{tabular}}}}%
  \end{picture}%
\endgroup%
}
\caption{\label{F:convex}  The polygon $\Omega$ near $p_0 \in \p
  \Omega$ when $p_0$ lies at a convex corner.  The domain $\Omega$ is
  to the right of $p_0$, so that the interior angle $\theta_0$ satisfies
  $0 < \theta_0 < \pi$.  The distance $d$
  is to the closest vertex or closest other side, whichever is
  closer.  Here $\p \Omega$
is in bold, and the sector $S$ of radius $d$ which is contained in
$\Omega$ is in
dashes. }
    \end{figure}

    \begin{figure}
\centering
    {
\begingroup%
  \makeatletter%
  \providecommand\color[2][]{%
    \errmessage{(Inkscape) Color is used for the text in Inkscape, but the package 'color.sty' is not loaded}%
    \renewcommand\color[2][]{}%
  }%
  \providecommand\transparent[1]{%
    \errmessage{(Inkscape) Transparency is used (non-zero) for the text in Inkscape, but the package 'transparent.sty' is not loaded}%
    \renewcommand\transparent[1]{}%
  }%
  \providecommand\rotatebox[2]{#2}%
  \newcommand*\fsize{\dimexpr\f@size pt\relax}%
  \newcommand*\lineheight[1]{\fontsize{\fsize}{#1\fsize}\selectfont}%
  \ifx\svgwidth\undefined%
    \setlength{\unitlength}{281.73835945bp}%
    \ifx\svgscale\undefined%
      \relax%
    \else%
      \setlength{\unitlength}{\unitlength * \real{\svgscale}}%
    \fi%
  \else%
    \setlength{\unitlength}{\svgwidth}%
  \fi%
  \global\let\svgwidth\undefined%
  \global\let\svgscale\undefined%
  \makeatother%
  \begin{picture}(1,0.95720523)%
    \lineheight{1}%
    \setlength\tabcolsep{0pt}%
    \put(0,0){\includegraphics[width=\unitlength,page=1]{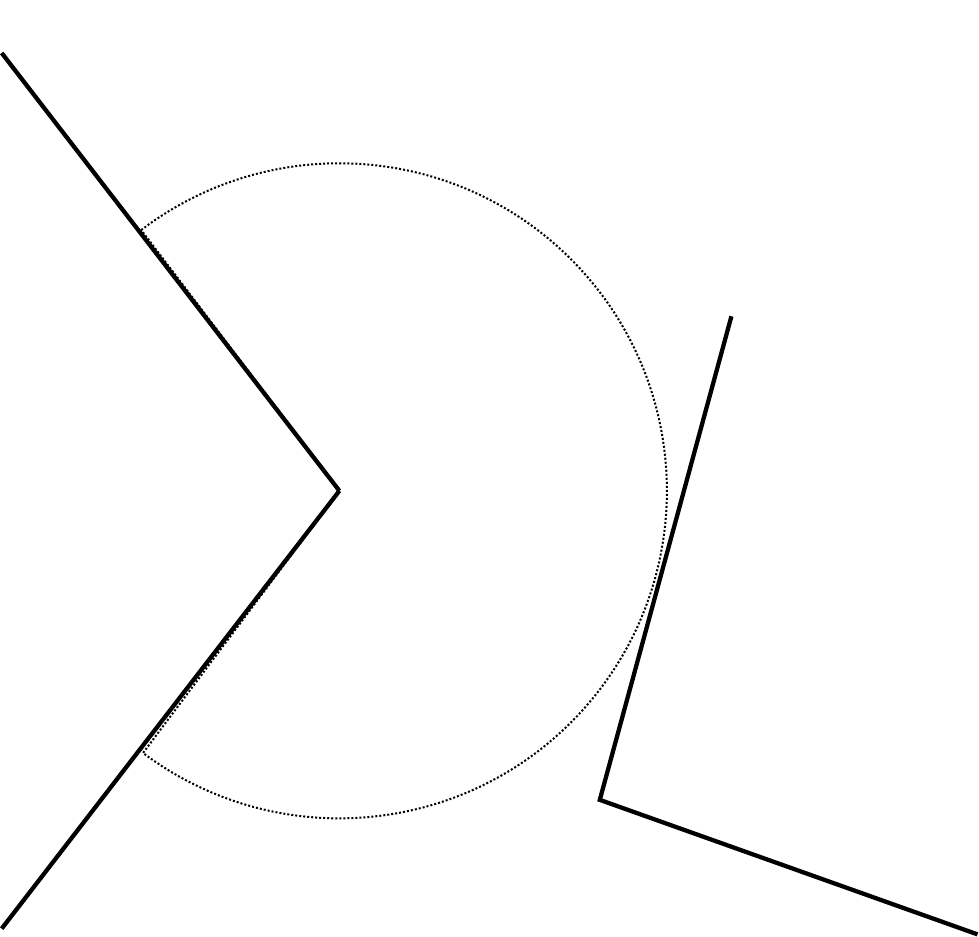}}%
    \put(0.4025007,0.52861785){\color[rgb]{0,0,0}\makebox(0,0)[lt]{\lineheight{1.25}\smash{\begin{tabular}[t]{l}$S$\end{tabular}}}}%
    \put(0.47270845,0.93231256){\color[rgb]{0,0,0}\makebox(0,0)[lt]{\lineheight{1.25}\smash{\begin{tabular}[t]{l}$\p \Omega$\end{tabular}}}}%
    \put(0.45574152,0.34717469){\color[rgb]{0,0,0}\makebox(0,0)[lt]{\lineheight{1.25}\smash{\begin{tabular}[t]{l}$d$\end{tabular}}}}%
    \put(0,0){\includegraphics[width=\unitlength,page=2]{poly-concave.pdf}}%
  \end{picture}%
\endgroup%
}
\caption{\label{F:concave}    The polygon $\Omega$ near $p_0 \in \p
  \Omega$ when $p_0$ lies at a concave corner.  The domain $\Omega$ is
  to the right of boundary segments adjacent to $p_0$, so that the interior angle $\theta_0$ satisfies
  $\pi < \theta_0 < 2 \pi$.  The distance $d$
  is to the closest vertex or closest other side, whichever is
  closer.  Here $\p \Omega$
is in bold, and the sector $S$ of radius $d$ which is contained in
$\Omega$ is in
dashes. }
    \end{figure}

    \begin{proposition}
\label{P:flat-side}
Let $u$ be a solution of \eqref{E:ef-2}, and let $p_0 \in \p \Omega$.
Let $d = \dist(p_0, F')$, where $F'$ is the nearest edge to $p_0$
which is not adjacent to $p_0$.  
Then  for each $0 < \alpha < 1$,  
\begin{equation}
  \label{E:flat-side-est}
  \limsup_{h \to 0+} \| u \|^2_{L^2 ( D(p_0, \alpha d ) \cap \Omega) } \leq
  \frac{1}{2-\alpha}.
  \end{equation}


    \end{proposition}

Before jumping into the proof, we need some knowledge of how mixed
boundary conditions interact with the radial vector field.

\begin{lemma}
\label{L:zero}

Let $p_0 \in F$ for some face $F$, and let $d = \dist(p_0,
F')$, where $F'$ is the  closest non-adjacent face to $p_0$.  Let $u$ be a solution to
\eqref{E:ef-2} subject to the boundary conditions $Bu = 0$.  Then for any $0 < R < d$ and
any radial $\phi \in \Ci_c (D(p_0, R))$,
\[
2 \int_\Omega \phi |u|^2 dV
 = \int_{\Omega} (Xu) (
    [-h^2 \Delta, \phi] \bu) dV .
    \]

  \end{lemma}

\begin{proof}

For a boundary point $p_0 \in \p \Omega$, $p_0$ either lies in the
interior of a flat face, or at a corner.  If at a corner, it can
either be a convex or concave corner.  But it turns out that the proof works
more or less the same for all three cases.  In fact, the proof is
nearly identical to the proof of the interior case.  Let $\theta_0$ be  the
angle of $\p \Omega$ at $p_0$, measured from the interior of
$\Omega$.  If $0 < \theta_0 < \pi$, $p_0$ lies at a convex corner, and
if $\pi < \theta_0 < 2 \pi$, $p_0$ lies at a concave corner.  If $\theta_0
= \pi$, then $p_0$ lies on the interior of a face of $\p \Omega$.

Translate so that $p_0 = 0$ and rotate so that locally near $0$
\[
\Omega = \{ (r, \theta) : 0 \leq r \leq R, \,\, -\theta_0/2 \leq
\theta \leq \theta_0/2 \}.
\]
This just means that locally $\Omega$ looks like a sector of a disc,
with $\Omega$ on the right hand side.

The proof procedes by considering the different boundary conditions
when $\theta_0 = \pi$ and when $\theta_0 \neq \pi$.  
The first computation uses Lemma \ref{L:reproduce}:
\[
\int_\Omega \phi ([-h^2 \Delta -1, X] u ) \bu dV = -2 \int_\Omega \phi
(-h^2 \Delta ) \bu dV = 2 \int_\Omega \phi |u|^2 dV.
\]
On the other hand, unpacking the commutator gives
\[
\int_\Omega \phi ((-h^2 \Delta -1) X u - X ( -h^2 \Delta -1) u ) \bu
dV = \int_\Omega \phi ((-h^2 \Delta -1) X u) \bu dV
\]
since $u$ satisfies \eqref{E:ef-2}.  Using Green's formula gives
\begin{align*}
  \int_\Omega & \phi ((-h^2 \Delta -1) X u) \bu dV \\
  & = \int_\Omega (Xu)
((-h^2 \Delta -1) \phi \bu )dV - \int_{\p \Omega} \phi (h \p_\nu hX u) \bu
  dS + \int_{\p\Omega} (hX u ) (h \p_\nu \phi \bu ) dS \\
  &= \int_\Omega (Xu)
([-h^2 \Delta  ,\phi] \bu )dV - \int_{\p \Omega} \phi (h \p_\nu hX u) \bu
  dS + \int_{\p\Omega} (hX u ) (h \p_\nu \phi \bu ) dS.
\end{align*}
Here the boundary integrals over $\p \Omega$ are really just over the
segments within the support of $\phi$.

We first consider
the case $\theta_0 = \pi$.
In this case, $u$ lives in a flat face, so has the same boundary
condition on the whole segment $F: = \{ x = 0 , \,\, -R \leq y \leq R \}$
in rectangular coordinates.  The outward unit normal derivative on $F$
is $- \p_x$.  If $u|_F = 0$, then
\[
 \int_{\p \Omega} \phi (h \p_\nu hX u) \bu
 dS = 0.
 \]
 If $u|_F = 0$, then $\p_y u |_F = 0$ as well, since it is a
 tangential derivative.  Writing
$X = x \p_x + y \p_y$, we have
\[
hX u|_{F} = (x \p_x u + y \p_y u )|_F = x \p_x u|_F = 0,
\]
since $F \subset \{ x = 0 \}$.  Hence
\[
\int_{\p\Omega} (hX u ) (h \p_\nu \phi \bu ) dS = 0,
\]
which proves the Lemma in this  case.

If $ \p_\nu u |_F = 0$, then
\[
\int_{\p\Omega} (hX u ) (h \p_\nu \phi \bu ) dS = 0.
\]
This is true since $\phi$ is assumed radial, so that $\p_\theta \phi =
0$, and along $F$ we have $\p_\nu = \p_\theta = - \p_x$.  
Since $\p_\nu = - \p_x$, that also means that $\p_x u|_F = 0$.  We now
compute:
\begin{align*}
  h\p_\nu hX u|_F & = -h \p_x ( x h \p_x u + y h \p_y u )|_F \\
  & = -h h \p_x u|_F - x h^2 \p_x^2 u |_F - y h \p_y h \p_x u |_F \\
  & = 0,
\end{align*}
since $\p_x u = 0$, $x = 0$, and $\p_y$ is purely tangential.

We now consider the case when $0 < \theta_0 < \pi$, or $\pi < \theta_0
< 2 \pi$
so that $p_0$ lies
at a convex corner, or concave corner respectively.   The upper (respectively lower) segments meeting
at $p_0$ can be parametrized by 
 $y = ax/b$ (respectively $y = -ax/b$).  In the case $0 <
\theta_0 < \pi$, $a, b>0$, while in the case $\pi < \theta_0 < 2 \pi$,
$a>0$ and $b<0$.  The reason for taking $b<0$ instead of $a<0$ here is
so that the same notation may be used for the normal derivatives in
both the convex and concave cases.  Let $F_1$ denote
the upper segment and $F_2$ denote the lower segment.  
There are three cases to
consider, corresponding to the three possible boundary conditions
(Dirichlet-Dirichlet, Neumann-Neumann, and Dirichlet-Neumann).  Of
course the Neumann-Dirichlet case is then obtained by reflection over
$\{ y = 0 \}$.  Our first task is to determine the correct outward
pointing normal derivatives, and a choice of tangent derivatives.  
Along $F_1$, a choice of tangent derivative is
\[
\p_\tau = \frac{b}{c} \p_x + \frac{a}{c} \p_y,
\]
and the outward normal derivative is
\[
\p_\nu = -\frac{a}{c} \p_x + \frac{b}{c} \p_y.
\]
Here we pause briefly to note that this is where it is convenient to
take $a>0$, $b<0$ in the concave case so that the normal derivative
points outward.  
Along $F_2$, we have
\[
\p_\tau = \frac{b}{c} \p_x - \frac{a}{c} \p_y,
\]
and
\[
\p_\nu = - \frac{a}{c} \p_x - \frac{b}{c} \p_y.
\]

In the case where $u |_{F_j} = 0$ for $j = 1,2$, we could use that $r
\p_r$ is tangential, but it is instructive to recall the argument from
\cite{Chr-tri} in rectangular coordinates.  If $u|_{F_1} = 0$, then
\[
0 = \p_\tau u|_{F_1} = \frac{b}{c} \p_x u|_{F_1} + \frac{a}{c} \p_y u
|_{F_1} ,
\]
so that
\[
\p_x u|_{F_1} = -\frac{a}{b} \p_y u |_{F_1}.
\]
That means that
\begin{align*}
\p_\nu u|_{F_1} & = \left(  \frac{a^2}{bc}  + \frac{b}{c}
\right) \p_y u|_{F_1}  \\
& =
\frac{c}{b} \p_y u|_{F_1}
,
\end{align*}
so that
\[
\p_yu|_{F_1} = \frac{b}{c} \p_\nu u|_{F_1},
\]
and
\[
\p_x u |_{F_1} = -\frac{a}{c} \p_\nu u|_{F_1}.
\]
Plugging into the vector field $X$, we have
\begin{align*}
  hX u|_{F_1} & = x h \p_x u|_{F_1} + y h \p_y u|_{F_1} \\
  & = -\frac{a}{c} x h\p_\nu u |_{F_1} + \left(\frac{ax}{b} \right) \frac{b}{c} h \p_\nu
  u|_{F_1}\\
  & = 0.
\end{align*} 
If $u|_{F_2} = 0$, a similar computation holds on $F_2$, so that
\[
- \int_{\p \Omega} \phi (h \p_\nu hX u) \bu
dS + \int_{\p\Omega} (hX u ) (h \p_\nu \phi \bu ) dS = 0,
\]
and the Lemma is proved for that case.

For the next case, let us assume $\p_\nu u|_{F_1} = 0$.  Then
\begin{align*}
  \int_{\p\Omega}  (hX u ) (h \p_\nu \phi \bu ) dS &  = \int_{\p \Omega} (hX u) (h \p_\nu \phi) \bu dS \\
  & = 0
\end{align*}
again since $\phi$ is assumed radial.  On the other hand, for the other boundary integral
term, we again compute:
\begin{align}
  h \p_\nu h X u|_{F_1} &= \left( - \frac{a}{c} h\p_x + \frac{b}{c} h\p_y
  \right) (x h \p_x u|_{F_1} + y h \p_y u|_{F_1} ) \notag \\
  & = -h \frac{a}{c} h \p_x u|_{F_1} + h \frac{b}{c} h \p_y u|_{F_1}
  \notag \\
  & \quad + x h \p_\nu h \p_x u|_{F_1} + 
  \left(\frac{ax}{b}\right) h \p_\nu h \p_y u|_{F_1}. \label{E:F1-Neumann}
\end{align}
Here we have simply computed the contribution when $h \p_\nu$ falls on
$x$ or on $y$ in $hX$.
Next, we  write $h \p_xu$
and $h \p_y u$ in terms of tangential and normal derivatives.  We have
\[
\p_x = \frac{b}{c} \p_\tau - \frac{a}{c} \p_\nu,
\]
and
\[
\p_y = \frac{a}{c} \p_\tau + \frac{b}{c} \p_\nu.
\]
Then
\begin{align*}
h \p_\nu h \p_x u|_{F_1} & = \frac{b}{c} h \p_\nu h \p_\tau u|_{F_1} -
\frac{a}{c} h^2 \p_\nu^2 u|_{F_1} \\
& = \frac{a}{c} ( 1 + h^2 \p_\tau^2  )u|_{F_1},
\end{align*}
since our  change of variables $(x,y) \to (\tau , \nu)$ then gives
\[
-h^2 \p_\nu^2 u - h^2 \p_\tau^2 u = u
\]
in a neighbourhood of $F_1$ and $\p_\tau \p_\nu u|_{F_1} = 0$.  Similarly,
\begin{align*}
h \p_\nu h \p_y u|_{F_1} & = \frac{a}{c} h \p_\nu h \p_\tau u|_{F_1} +
\frac{b}{c} h^2 \p_\nu^2 u|_{F_1} \\
& = -\frac{b}{c} (1 + h^2 \p_\tau^2 ) u|_{F_1}.
\end{align*}
Plugging in to \eqref{E:F1-Neumann}, we have
\begin{align*}
  h \p_\nu h X u|_{F_1} & =
 -h \frac{a}{c} h \p_x u|_{F_1} + h \frac{b}{c} h \p_y u|_{F_1}
 \\
  & \quad  x \frac{a}{c} ( 1 + h^2 \p_\tau^2  )u|_{F_1} 
 -\frac{b}{c}
 \left(\frac{ax}{b} \right) (1 + h^2 \p_\tau^2 ) u|_{F_1} \\
 & =  -h \frac{a}{c} h \p_x u|_{F_1} + h \frac{b}{c} h \p_y u|_{F_1}.
\end{align*}
For the two remaining ``lower order'' terms, it is tempting to
integrate and try to apply known restriction estimates.  However, the
lack of smoothness of the boundary makes this very delicate.  In any
case, we can still write these terms in terms of the normal and
tangent derivatives, recalling that we are assuming that $\p_\nu u =
0$.  That means that
\begin{align*}
h \p_x u|_{F_1} &  = \left( \frac{b}{c}h \p_\tau - \frac{a}{c} h\p_\nu
\right) u|_{F_1} \\
& = \frac{b}{c}h \p_\tau u|_{F_1}.
\end{align*}
Similarly,
\[
h \p_y u|_{F_1} = \frac{a}{c} h \p_\tau u|_{F_1}.
\]
Then 
\begin{align*}
 -h \frac{a}{c} h \p_x u|_{F_1} + h \frac{b}{c} h \p_y u|_{F_1} & = h
 \left( -\frac{a}{c} \left( \frac{b}{c}h \p_\tau u|_{F_1} \right) +
 \frac{b}{c} \left( \frac{a}{c} h \p_\tau u|_{F_1} \right) \right) \\
 & = 0.
 \end{align*}
Again, a similar computation holds on $F_2$, so the Lemma is proved
for the case where $\p_\nu u = 0$ on both $F_1$ and $F_2$.  Of course
these analyses are completely independent of each other, so work for
mixed boundary conditions as well.  This proves the Lemma for $p_0$ at
a 
corner, completing the proof.

\end{proof}

We now continue with the proof of the Proposition.

    \begin{proof}[Proof of Proposition \ref{P:flat-side}]
      The proof is by contradiction, and proceeds almost verbatim from
      the interior case.  We only point out the differences here.

Let $d = \dist(p_0, F')$, where $F'$ is the nearest edge to $p_0$
which is not adjacent to $p_0$ as in the Proposition.  This is either the nearest corner to $p_0$, or a
different side if it is closer than the nearest corner.  By rotating
and translating $\Omega$ we may assume that $p_0 = 0$.  We may also
assume that $\Omega$ is locally symmetric over the line $y = 0$, and
locally the segment $\{ x >0, \, y = 0 \} \subset \Omega$.  In the case $p_0$ is in the
interior of a face, then the face is vertical and $\Omega$ locally
lies to the right of that face.  If $p_0$ lies at a convex corner
making interior angle $\theta_0$, then $\Omega$ lies locally to the
right  of $p_0$ and locally the angles above and below $\{ y = 0 \}$
are $\theta_0/2$.  If $p_0$ lies at a concave corner, we use the same
setup with angle $\theta_0/2$ above and below the line $\{ y = 0 \}$,
but now $\Omega$ lies to the right of the segments at angle $\pm
\theta_0/2$.  See Figures \ref{F:flat}-\ref{F:concave}.

Let $d, \epsilon, \phi, \psi$ all be exactly the same as in the proof
of Proposition \ref{P:interior}.
Every computation is exactly the
same as in the proof for Proposition \ref{P:interior}, except for the
integrations by parts, since now there are 
potentially boundary terms.  We now apply Lemma \ref{L:zero} to compute
\begin{align*}
  2\int_\Omega  \phi |u|^2 dV 
  & =   \int_\Omega  \phi ([ -h^2 \Delta -1, r \p_r ] u )\bu dV\\
  & = \int_\Omega  (Xu) ([ -h^2 \Delta , \phi ] \bu ) dV.
\end{align*}
The rest of the proof follows precisely the proof of Proposition \ref{P:interior}.

\end{proof}


\appendix

\section{Construction of the function $\phi_1$}
In the proofs of Propositions \ref{P:interior} and \ref{P:flat-side}, we have used a cutoff function with nice properties.
The construction of such a function is more or less well-known, but we
want very precise estimates, so we discuss the existence of such a cutoff in the next Lemma.

\begin{lemma}
  \label{L:phi1}
Fix $0 < \delta_1 < \delta_2$ and $0 < \epsilon \leq
\min((\delta_2-\delta_1)/4, 1/2)$.  There exists a function $\phi_1
\in \Ci( \reals)$ satisfying the
following conditions:
\begin{enumerate}
  \item $\phi_1 \geq 0$   and $\phi' \leq 0$, 
\item $\phi_1(s) \equiv 1$ for $s \leq \delta_1 + \epsilon^3$,
\item $\phi_1 (s) \equiv 0$ for $s \geq  \delta_2 - \epsilon^3$, and
  \item $| \phi_1' | \leq \frac{1}{\delta_2-\delta_1} + \epsilon$.

  \end{enumerate}

\end{lemma}

\begin{proof}
The only non-trivial part is the last condition on the derivative,
since $\phi_1'$ is supported in a set strictly smaller than size
$\delta_2 - \delta_1$.
The idea is that we can make $\phi_1$ as close to linear as we want
and the only thing to work out is the dependence on the parameter
$\epsilon>0$.  We recall that for any small
number $\eta>0$, there exists a smooth function $\phi_2: \reals \to \reals$
with $\phi_2(s) \equiv 1$ for $s \leq 0$, $\phi_2(s) \equiv 0$ for $s
\geq 1$, $\phi_2' \leq 0$, and $| \phi'_2 | \leq 1 + \eta$.
For our $\epsilon>0$, choose such a $\phi_2$ with $\eta = \epsilon^2$.

In order to simplify notation, let $\delta = \delta_2 - \delta_1 >0$.  
Let
\[
\phi_1(s) = \phi_2\left(   \frac{s  - (\delta_1 + \epsilon^3)}{\delta -
  2 \epsilon^3} \right),
\]
so that $\phi_1$ satisfies
$\phi_1 (s) \equiv 1$ for $s \leq \delta_1 + \epsilon^3$, $\phi_1(s)
\equiv 0$ for $s \geq  \delta_2 - \epsilon^3$, $\phi_1' \leq 0$, and
\begin{align*}
\sup  | \phi_1' | & \leq \frac{1}{\delta - 2 \epsilon^3} \sup |
\phi_2' | \\
& \leq \frac{1 + \epsilon^2}{\delta - 2 \epsilon^3}.
\end{align*}



We are going to make a geometric series type expansion for which we
need an upper bound.  Recall that
for $f(t) = (1-t)^{-1}$, $t\geq 0$ small, we have
$f(t)  = 1 + t f'(s)$ for some $0 \leq s \leq t$.  As $f'(s) = (1 -
s)^{-2}$, we know that for $0 \leq s \leq t$, $| f'(s) | \leq (1 -
t)^{-2}$.  Hence
\begin{equation}
  \label{E:taylor}
f(t) \leq 1 + t(1-t)^{-2}.
\end{equation}

Now we have assumed that $\epsilon \leq \min(\delta/4, 1/2)$, so we have
\[
\frac{2\epsilon^3}{\delta} \leq \frac{\epsilon^2}{2} \leq
\frac{1}{2},
\]
so that
\[
\frac{1}{1 - 2 \epsilon^3/\delta} \leq 2.
\]

Plugging this into \eqref{E:taylor} with $t = 2 \epsilon^3/\delta$, we
have 
\begin{align*}
  \frac{1 + \epsilon^2}{\delta - 2 \epsilon^3} & = \frac{1 +
    \epsilon^2}{\delta} \left( \frac{1}{1-2 \epsilon^3/\delta} \right) \\
  & \leq \frac{1 +
    \epsilon^2}{\delta} \left( 1 + \frac{2
    \epsilon^3}{\delta}\left(\frac{1}{1-2 \epsilon^3/\delta}
  \right)^2\right) \\
  & \leq \frac{1 +
    \epsilon^2}{\delta} \left( 1 + \frac{8
    \epsilon^3}{\delta}\right).
\end{align*}
Continuing, and using that $\epsilon \leq \min ( \delta/4, 1/2)$, we have
\begin{align}
  \frac{1 +
    \epsilon^2}{\delta} \left( 1 + \frac{8
    \epsilon^3}{\delta}\right) 
  & \leq ({1 + \epsilon^2}){}\left( \frac{1}{\delta} + \frac{8
    \epsilon^3}{\delta^2} \right) \notag \\
  & \leq (1 + \epsilon^2) \left( \frac{1}{\delta} + \frac{\epsilon}{2}
  \right) \notag \\
  & = \frac{1}{\delta} + \frac{\epsilon}{2} + \frac{\epsilon^2}{\delta} +
  \frac{\epsilon^3}{2} \notag \\
  & \leq \frac{1}{\delta} + \epsilon \notag \\
  & = \frac{1}{\delta_2 - \delta_1} + \epsilon. \label{E:phi-prime}
  \end{align}

\end{proof}

\bibliographystyle{alpha}
\bibliography{HC-bib}

\end{document}